\documentclass[12pt]{amsart}
\usepackage{amssymb,verbatim,amscd,amsmath,graphicx,enumerate}
\usepackage{graphicx}
\usepackage{caption}
\usepackage{subcaption}
\usepackage{pdfsync}
\usepackage{fullpage}
\usepackage{color}
\usepackage{marginnote}
\usepackage{tikz}
\usepackage{bbm}
\usepackage{hyperref}

\numberwithin{equation}{section}
\newtheorem{theorem}{Theorem}[section]
\newtheorem{lemma}[theorem]{Lemma}

\newtheorem{corollary}[theorem]{Corollary}

\theoremstyle{definition}
\newtheorem{definition}[theorem]{Definition}
\newtheorem{def-prop}[theorem]{Definition-Proposition}
\newtheorem{conjecture}[theorem]{Conjecture}

\newtheorem{remark}[theorem]{Remark}
\newtheorem{example}[theorem]{Example}

\newtheorem*{acknowledgement}{Acknowledgement}

\newtheorem{algorithm}[theorem]{Algorithm}

\DeclareMathOperator{\rank}{rank}
\DeclareMathOperator{\bi}{bi}

\newcommand{\NN}{{\mathbb N}}

\newcommand{\kk}{{\mathbbm k}}
\def\T{{\mathbf T}}

\def\B{{\mathcal B}}

\def\R{{\mathcal R}}

\def\R{{\mathcal R}}

\def\mm{{\mathfrak m}}

\def\1{{\bf 1}}
\def\0{{\bf 0}}

\begin{document}
	
	\title{Algebraic algorithms for even circuits in graphs}
	
	\author{Huy T\`ai H\`a}
	\address{Department of Mathematics \\
		Tulane University \\
		6823 St. Charles Avenue \\
		New Orleans, LA 70118}
	\email{tha@tulane.edu}
	\urladdr{http://www.math.tulane.edu/~tai/}
	
	\author{Susan Morey}
	\address{Department of Mathematics \\
		Texas State University\\
		601 University Drive\\
		San Marcos, TX 78666}
	\email{morey@txstate.edu}
	\urladdr{http://www.txstate.edu/~sm26/}

	
	\begin{abstract}
		We present an algebraic algorithm to detect the existence of and to list all indecomposable even circuits in a given graph. We also discuss an application of our work to the study of directed cycles in digraphs.
	\end{abstract}
	
	\maketitle
	

\section{Introduction} \label{sec.intro}

Detecting the existence of cycles in graphs is a fundamental problem in graph theory (cf. \cite{AYZ1, AYZ2, CKS, CCKV, C, DKS, FZ, HOV1, HOV2, MM, M, Seymour, T, YZ, Z}). Graph theoretic algorithms exist to enumerate both odd and even cycles. In \cite{FHVT}, the first author, together with Francisco and Van Tuyl, gave an algebraic algorithm to detect and exhibit all induced odd cycles in an undirected graph. The study of \cite{FHVT} is an example of the rich interaction between commutative algebra and graph theory. In fact, using algebraic methods to study combinatorial structures and using combinatorial data to understand algebraic properties and invariants has evolved to be an active research topic in combinatorial commutative algebra in recent years (cf. \cite{FHM, MV} and references therein).

In the present paper, we continue this line of work and describe an algebraic algorithm to enumerate even circuits in an undirected graph; a \emph{circuit} is a closed walk in which the edges are all distinct and a \emph{cycle} is a closed walk in which the vertices are all distinct. We shall also discuss an application of our work to the problem of finding directed cycles in a directed graph (digraph).
Let $G = (V,E)$ be a finite simple undirected graph on the vertex set $V = \{x_1, \dots, x_n\}$. Let $\kk$ be a field and identify the vertices in $V$ with the variables in $R = \kk[x_1, \dots, x_n]$. The \emph{edge ideal} of $G$ is defined to be
$$I(G) = \langle x_ix_j ~\big|~ x_ix_j \in E\rangle.$$
The construction of edge ideals of graphs was first introduced by Villarreal in \cite{Rafael1990} (see also \cite{F, HVT} for edge ideals of simplicial complexes and hypergraphs) and has been an essential tool in various studies in this area of research. Our main result states that even circuits in $G$ can be detected by considering the \emph{reduced Jacobian dual} of the edge ideal $I(G)$, a notion which we will defined in Section \ref{sec.prel}.

An even circuit is called \emph{indecomposable} if it cannot be realized as the edge-disjoint union of two smaller even circuits. Our first theorem reads as follows, leaving unexplained terminology until later (Section \ref{sec.prel}).

\medskip

\noindent{\bf Theorem \ref{thm.main}.} Let $G$ be a graph, $I=I(G)$, and $\phi$ the presentation matrix from the Taylor resolution of $I$. Then the indecomposable even circuits of $G$ correspond exactly to the binomial minors of the reduced Jacobian dual $\overline{B(\phi)}$ which satisfy the following conditions:
\begin{enumerate}
\item the monomials in these binomials are square-free and relatively prime; and
\item the columns of the corresponding submatrices are pairwise \emph{center-distinct}.
\end{enumerate}

We focus on even circuits because they form a larger class than that of even cycles. With a slight modification of condition (2), we can also obtain an algebraic characterization for cycles of even lengths in $G$, see Remark \ref{rmk.cycle}. The proof of Theorem \ref{thm.main} is based on an ad hoc analysis of the possible forms of minors of the reduced Jacobian dual $\overline{B(\phi)}$.

Theorem \ref{thm.main} allows us to derive an algebraic algorithm to enumerate all indecomposable even circuits in a given graph that runs in polynomial time on the size of the edge set of the graph, see Algorithm \ref{alg}. Our goal is not to compare the running time of our algorithm to that of existed ones, rather we aim to exhibit yet another interesting connection between commutative algebra and graph theory. Theorem \ref{thm.main} also has an algebraic consequence to finding defining equations for the Rees algebras of edge ideals of graphs, see Theorem \ref{thm.J}.

Theorem \ref{thm.main} furthermore has an interesting application toward the study of directed cycles in digraphs. For a digraph $D$, we construct a bipartite graph $G = G(D)$, see Definition \ref{def.Bgraph}. Note that this bipartite graph has a natural perfect matching, which we denote by $M_D$. There is an established equivalence between the directed cycles in $D$ and the even cycles in $G$ with a certain property that traces its roots back to work done by Dulmage and Mendelsohn in the 1950's (see for example \cite{DM, KL}) which we restate for convenience. Specifically, again leaving unexplained terminology until later, we have:

\medskip

\noindent{\bf Theorem \ref{thm.cycle}.} The directed cycles in a digraph $D$ correspond exactly to the even cycles in $G = G(D)$ in which a collection of alternating edges forms a subset of the perfect matching $M_D$.

\medskip

Theorem \ref{thm.cycle}, combined with Algorithm \ref{alg}, gives an algebraic algorithm to enumerate all directed cycles in digraphs, see Corollary \ref{cor.digraph}. As a consequence of Theorems \ref{thm.main} and \ref{thm.cycle}, we are also able to translate the famous Caccetta-H\"aggkvist conjecture for directed cycles in digraphs to a statement about binomial minors of the Jacobian dual matrix, see Conjecture \ref{conj}.

\begin{acknowledgement} The first author is partially supported by Louisiana Board of Regents (grant \#LEQSF(2017-19)-ENH-TR-25). The authors thank the anonymous referees for a careful, detailed reading of the manuscript and useful suggestions.
\end{acknowledgement}


\section{Preliminaries} \label{sec.prel}

In this section, we collect important notations and definitions used in the paper. For unexplained terminology in commutative algebra, we refer the reader to \cite{HH, Villarreal2015}, and in graph theory, we refer the reader to \cite{D}.

\medskip

\noindent{\bf Algebra.} Throughout the paper, $\kk$ denotes an infinite field. Let $R = \kk[x_1, \dots, x_n]$ be a polynomial ring over $k$ and let $\mm = (x_1, \dots, x_n)$. Let $I \subseteq R$ be an ideal and use $\mu(I)$ to denote the minimal number of generators of $I$. Let $\phi$ be a presentation matrix of $I$.


\begin{definition}\label{Rees}
	The \emph{Rees algebra} of $I$ is defined to be the graded ring
	$$\R(I) = R[It] = \R \oplus It \oplus I^2t^2 \oplus \dots \subseteq R[t].$$

Suppose that $I = (f_1, \dots, f_r)$. Then there exists a natural presentation of the Rees algebra of $I$, namely,
$$R[T_1, \dots, T_r] \stackrel{\theta}{\rightarrow} R[It] \rightarrow 0,$$
given by $T_i \mapsto f_it$ for $i = 1, \dots, r$, where $T_1, \dots, T_r$ are indeterminates. Set $J= {\mbox{\rm ker }}\theta$. Then $R[It] \cong R[T_1, \ldots, T_r]/J$, and $J$ is referred to as the {\em ideal of equations} or {\em defining ideal} of $R[It]$. Since $\phi(T_i)=f_it$, we will say that $T_i$ corresponds to the generator $f_i$ of $I$.
\end{definition}

By the definition of a presentation matrix $\phi$, the linear (in the variables $T_1, \ldots, T_r$) equations of $J$ are generated by entries of the matrix $[T_1 \ \dots \ T_r]\cdot \phi$. When these entries are linear in $x_1, \ldots, x_n$, that is, when the entries of $\phi$ are linear, then $\phi$ is the Jacobian matrix of these equations with respect to $T_1, \ldots, T_r$. In this setting, one can also define another Jacobian matrix of the same polynomials in $[T_1 \ \dots \ T_r]\cdot \phi$ but with respect to $x_1, \dots, x_n$. This new Jacobian matrix is usually denoted by $B(\phi)$ and referred to as the {\em Jacobian dual} of $\phi$. We now give the generalized version of this notion when the entries of $\phi$ are not necessarily all linear. See \cite[Section 1.5]{Wolmer} and \cite{Morey} for further information on Jacobian duals.

\begin{definition}[Jacobian dual] \label{def.dual}
	Let $r = \mu(I)$ and let $\phi$ be a presentation matrix of $I$ with respect to a set of $r$ generators of $I$.
	\begin{enumerate}
		\item A \emph{Jacobian dual} of $\phi$, denoted by $B(\phi)$, is defined to be a matrix, whose entries are in $R[T_1, \dots, T_r]$ and linear in the variables $T_1, \dots, T_r$, that satisfies the equation
		$$[T_1 \dots T_r]\cdot \phi = [x_1 \dots x_n]\cdot B(\phi).$$
		\item The \emph{reduced Jacobian dual} of $\phi$, denoted by $\overline{B(\phi)}$, is defined to be $B(\phi) \otimes_k R/\mm$.
	\end{enumerate}
\end{definition}
Observe that given a fixed $\phi$, there may be more than one choice for $B(\phi)$, but $\overline{B(\phi)}$ exists uniquely up to elementary row and column rearrangements that come from re-orderings (see, for example, \cite{Morey}).
The matrix $B(\phi)$, or $\overline{B(\phi)}$, has served as a source for the higher degree generators of $J$, see \cite{JM, KPU, Morey, UV, W} for example, with the emphasis being on minors of $\overline{B(\phi)}$.

\begin{example}\label{ex.twosquares}
Consider the graph

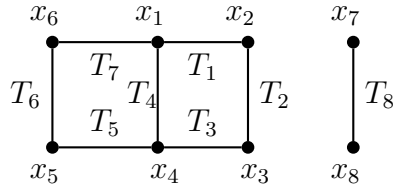
\begin{center}
\begin{tikzpicture}

\tikzstyle{point}=[circle,thick,draw=black,fill=black,inner sep=0pt,minimum width=4pt,minimum height=4pt]

 \node (a)[point,label={[xshift=-0.1cm, yshift=-0.7 cm]:$x_5$}] at (.5,0) {};

\node (b)[point,label={[xshift=-0.1cm, yshift=0.0 cm]:$x_6$}] at (0.5,1.4) {};

\node (f)[point,label={[xshift=0.1cm, yshift=-0.7 cm]:$x_4$}] at (1.9,0) {};

\node (e)[point,label={[xshift=0.1cm, yshift=-0.7 cm]$x_3$}] at (3.1,0) {};

\node (d)[point,label={[xshift=-0.1cm, yshift=-0 cm]:$x_2$}] at (3.1,1.4) {};

\node (c)[point,label={[xshift=-0.1cm, yshift=-0 cm]:$x_1$}] at (1.9,1.4) {};

\node (g)[point,label={[xshift=-0.1cm, yshift=-0 cm] $x_7$}] at (4.5, 1.4){};

\node (h)[point,label={[xshift=-0.1cm, yshift=-0.7cm] $x_8$}] at (4.5, 0){};

\path[draw, thick]
(a) edge node[auto] {$T_6$} (b)
(a) edge node[above] {$T_5$} (f)
(c) edge node[below] {$T_1$} (d)
(c) edge node[xshift=-0.2cm] {$T_4$} (f)
(d) edge node[auto] {$T_2$} (e)
(b) edge node[below] {$T_7$} (c)
(e) edge node[above]{$T_3$} (f)
(g) edge node[auto] {$T_8$} (h);

\end{tikzpicture}
\end{center}
corresponding to $I=(x_1x_2, x_2x_3,x_3x_4, x_1x_4,x_4x_5, x_5x_6, x_6x_1, x_7x_8).$
 Then,
 \begin{tiny}
  $$\phi = \left[ \begin{array}{ccccccccccccccccccccc}
 x_3 & 0 & 0 & x_4 & 0 & x_6 & 0 & 0 & 0 & 0 & x_3x_4 & x_4x_5 & x_5x_6 & x_7x_8 & \\
-x_1& x_4& 0 & 0& 0 & 0& 0 & 0 & 0 & 0 & 0 & 0 & 0 & 0 &\\
0 & -x_2 & -x_1 & 0 & -x_5 & 0 & 0 & 0 & 0 & 0 & -x_1x_2 & 0 & 0 & 0 & \\
 0 & 0& x_3 & -x_2& 0 & 0 & -x_5 & 0 & 0 & -x_6& 0 & 0 & 0 & 0 & \cdots \\
 0 & 0 & 0& 0&x_3& 0 & x_1 & -x_6 & 0 & 0 & 0 & -x_1x_2 & 0 & 0 & \\
 0 & 0 & 0 & 0 & 0& 0 & 0 & x_4& x_1 & 0& 0 & 0 & -x_1x_2 & 0 & \\
 0 & 0 & 0 & 0 & 0 & -x_2 & 0 & 0 & -x_5& x_4 & 0 & 0 & 0 & 0 & \\
 0 & 0 & 0 & 0 & 0 & 0 & 0 & 0 & 0 & 0 & 0 & 0 & 0 & -x_1x_2 & \\
 \end{array}\right]$$
 \end{tiny}
where the remaining columns of $\phi$ correspond to the rest of the (quadratic) Koszul relations on disjoint pairs of edges. The Koszul relations involving $T_1$ have been included for illustration.
Now,
\begin{eqnarray*}
\lefteqn{[T_1, \ldots, T_8] \cdot \phi =} \\
& & ( x_3T_1-x_1T_2, x_4T_2-x_2T_3, x_3T_4-x_1T_3, x_4T_1-x_2T_4, x_3T_5-x_5T_3, x_6T_1-x_2T_7, \\
& & x_1T_5 - x_5T_4, x_4T_6 - x_6T_5, x_1T_6 - x_5T_7, x_4T_7 - x_6T_4, x_3x_4T_1 - x_1x_2 T_3, \\
& & x_4x_5T_1 - x_1x_2 T_5, x_5x_6T_1-x_1x_2T_6, x_7x_8T_1 - x_1x_2 T_8, \cdots ).
\end{eqnarray*}
 When using these equations to form $B(\phi)$ as in Definition~\ref{def.dual}, the nonlinear terms (in the variables $x_i$'s) are ambiguous. For example, $x_3x_4T_1$ can be viewed as $x_3(x_4T_1)$ or as $x_4(x_3T_1)$. Different choices of $B(\phi)$ arise from different interpretations for each such nonlinear term in the $x_i$'s. The coefficient of $x_i$ of the $jth$ equation goes in the $(i,j)$ entry of $B(\phi)$. One such choice of $B(\phi)$ is
 \begin{tiny}
$$ B(\phi) = \left[ \begin{array}{ccccccccccccccccccccc}
 -T_2 & 0 & -T_3 & 0 & 0 & 0 & T_5 & 0 & T_6 & 0 & -x_2T_3 & -x_2T_5 & -x_2T_6 & 0 \\
 0 & -T_3 & 0 & -T_4 & 0 & -T_7 & 0 & 0 & 0 & 0 & 0 & 0 & 0 &  -x_1 T_8\\
 T_1 & 0 &  T_4 & 0 & T_5 & 0 & 0 & 0 & 0 & 0 & x_4T_1 & 0 & 0 & 0 & \\
 0 & T_2 & 0 & T_1 & 0 & 0 & 0 & T_6 & 0 & T_7 & 0 & x_5T_1 & 0 & 0 & \\
 0 & 0 & 0 & 0 & -T_3 & 0 & -T_4 & 0 & -T_7 & 0 & 0 & 0 & x_6T_1 & 0 & \cdots\\
 0 & 0 & 0 & 0 & 0 & T_1 & 0 & -T_5 & 0 & -T_4 &  0 & 0 & 0 & 0 & \\
 0 & 0 & 0 & 0 & 0 & 0 & 0 & 0 & 0 & 0 & 0 &0 & 0 & x_8T_1 &\\
 0 & 0 & 0 & 0 & 0 & 0 & 0 & 0 & 0 & 0 &  0 & 0 & 0 & 0 & \\
 \end{array}\right].$$
 \end{tiny}
Tensoring with $R/{\mm}$ yields
  $$\overline{B(\phi)} = \left[ \begin{array}{ccccccccccccccccccccc}
 -T_2 & 0 & -T_3 & 0 & 0 & 0 & T_5 & 0 & T_6 & 0 & 0 & 0 & 0& 0 \\
 0 & -T_3 & 0 & -T_4 & 0 & -T_7 & 0 & 0 & 0 & 0 & 0 & 0 & 0 &  0\\
 T_1 & 0 &  T_4 & 0 & T_5 & 0 & 0 & 0 & 0 & 0 & 0 & 0 & 0 & 0 & \\
 0 & T_2 & 0 & T_1 & 0 & 0 & 0 & T_6 & 0 & T_7 & 0 & 0 & 0 & 0 & \\
 0 & 0 & 0 & 0 & -T_3 & 0 & -T_4 & 0 & -T_7 & 0 & 0 & 0 & 0 & 0 & \cdots\\
 0 & 0 & 0 & 0 & 0 & T_1 & 0 & -T_5 & 0 & -T_4 &  0 & 0 & 0 & 0 & \\
 0 & 0 & 0 & 0 & 0 & 0 & 0 & 0 & 0 & 0 & 0 &0 & 0 & 0 &\\
 0 & 0 & 0 & 0 & 0 & 0 & 0 & 0 & 0 & 0 &  0 & 0 & 0 & 0 & \\
 \end{array}\right].$$
 \end{example}

Note that entries of $B(\phi)$ that result from interpretations of the nonlinear terms of $\phi$ become zero when passing to $\overline{B(\phi)}$. Thus, the nonzero columns of $\overline{B(\phi)}$ correspond precisely to the linear columns of $\phi$ and the $0$-rows of $\overline{B(\phi)}$ correspond to vertices that do not appear as endpoints of any path of length two in $G$. Such vertices can be isolated, part of a connected component consisting of a single edge, or the center vertex of a connected component that is a tree of diameter $2$. Deleting zero-rows and zero-columns will not change the minors of a matrix. Thus, in practice, when focusing on minors, one can work with a smaller matrix $\phi'$ defined by the linear columns of $\phi$, and assume that the content ideal, $I_1(\phi')$, is generated by a subset of the variables, say $x_1, \ldots, x_d$. In this case, $[T_1, \ldots, T_r] \cdot \phi' = [x_1, \ldots, x_d] \cdot \overline{B(\phi)}$.

When $I$ is a monomial ideal, a particular presentation matrix $\phi$ of $I$ that we shall make use of comes from the Taylor resolution of $I$. We shall now recall the construction of the Taylor resolution and its presentation matrix; see \cite{BPS, Taylor} for more details. For a collection $\B = \{f_1, \dots, f_r\}$ of polynomials in $R$ and a subset $\sigma \subseteq \B$, let $f_\sigma$ denote the least common multiple of $\{f_i ~\big|~ i \in \sigma\}$.

\begin{definition}\label{def.Taylor}
Let $I \subseteq R$ be a monomial ideal with the unique set of monomial generators $\B = \{f_1, \dots, f_r\}$. The \emph{Taylor resolution} of $I$ is the following complex:
$$0 \rightarrow F_r \stackrel{\partial_r}{\rightarrow} F_{r-1} \stackrel{\partial_{r-1}}{\rightarrow} \dots \stackrel{\partial_2}{\rightarrow} F_1 \stackrel{\partial_1}{\rightarrow} I \rightarrow 0,$$
where, for $p = 1, \dots, r$, $F_p = R^{r \choose p}$ is the free $R$-module of rank ${r \choose p}$ whose basis corresponds all subsets of $p$ elements from $\B$, and the differential map $\partial_p: F_p \rightarrow F_{p-1}$ is defined, for each basis element $e_\sigma \in F_p$ corresponding to a subset $\sigma$ of cardinality $p$ in $\B$, by
$$\partial_p(e_\sigma) = \sum_{f_\ell \in \sigma} (-1)^{\left|\{f_j \in \sigma ~|~ j < \ell\}\right|} \dfrac{f_\sigma}{f_{\sigma \setminus \{f_\ell\}}}e_{\sigma \setminus \{f_\ell\}}.$$
The \emph{presentation matrix} of $I$ from its Taylor resolution is the matrix corresponding to the map $F_2 \stackrel{\partial_2}{\rightarrow} F_1$; its $(\{f_j\}, \tau)$-entry, for $\{f_j\} \subseteq \B$, $\tau \subseteq \B$ with $|\tau| = 2$, is equal to 0 if $f_j \not\in \tau$, equal to $(-1)\dfrac{f_\tau}{f_j}$ if $\tau = \{f_j,f_k\}$ and $j < k$, and equal to $\dfrac{f_\tau}{f_j}$ if $\tau = \{f_j,f_k\}$ and $j > k$.
\end{definition}

The matrix $\phi$ in Example \ref{ex.twosquares} is an instance of the presentation matrix that comes from the Taylor resolution of a monomial ideal. Another important notion that we shall use is minors and ideals of minors of a matrix. 

\begin{definition} \label{def.minor}
Let $A$ be an $r \times s$ matrix whose entries are polynomials in $R$. For $t \le \min \{r, s\}$, a \emph{$t \times t$ minor} of $A$ is the determinant of a $t \times t$ submatrix of $A$. The ideal in $R$ generated by all $t \times t$ minors of $A$ is often denoted by $I_t(A)$. A minor is \emph{binomial} if it can be written as the sum (or difference) of two monomials in $R$.
\end{definition}

\medskip

\noindent{\bf Graph theory.} An \emph{undirected graph} $G = (V,E)$ consists of a set $V$ of distinct points, called the \emph{vertices}, and a collection $E$ of unordered pairs of vertices, called the \emph{edges}. We shall assume that all graphs in this paper are \emph{simple}; that is, a graph will have neither loops nor multiple edges. We shall write $xy$ for the undirected edge between vertices $x$ and $y$ in a graph.

\begin{definition} Let $G$ be an undirected graph.
\begin{enumerate}
\item A \emph{walk} is an alternating sequence of vertices and edges $x_1, e_1, x_2, e_2, \dots, e_{s-1}, x_s$ such that $e_{i} = \{x_i, x_{i+1}\}$ for all $i = 1, \dots, s-1$. Such a walk is said to be \emph{closed} if $x_1 = x_s$.
\item A walk is called a \emph{trail} if its edges are distinct (while its vertices may repeat). A closed trail is called a \emph{circuit}.
\item A walk is called a \emph{path} if its edges and vertices are distinct (except possibly at $x_1 = x_s$). A closed path is called a \emph{cycle}.
\item The \emph{length} of a walk is the number of edges that the walk transverses (including multiplicities). A walk is \emph{even} (or \emph{odd}) if its length is an even (or odd) number.
\end{enumerate}
\end{definition}

We often list only the vertices to indicate a walk since the edges are obvious from the vertices.
The main graph-theoretic structure that our work captures in this paper is indecomposable even circuits, which we shall define below. We also recall a similar notion of primitive even closed walks.

\begin{definition}
Let $G$ be a graph.
\begin{enumerate}
	\item An even circuit is \emph{indecomposable} if it cannot be realized as the edge-disjoint union of two smaller even circuits.
	\item An even closed walk is \emph{primitive} if it does not contain an even closed subwalk.
\end{enumerate}
\end{definition}

\begin{remark}
	\label{cycles.eqtns}
	There is a close connection between even closed walks in a graph $G$ and the equations of the Rees algebra of the edge ideal of $G$, see \cite{Rafael1995}. In particular, suppose $x_1, e_1, x_2, e_2, \dots, e_{2s-1}, x_{2s}, e_{2s}$, where $e_{i} = \{x_i, x_{i+1}\}$ and $e_{2s}=\{x_{2s}, x_1\}$, is an even closed walk and $\theta(T_i)=e_i$ as in Definition~\ref{Rees}. Then 
	$$\theta(\prod_{i=1}^s T_{2i} - \prod_{i=1}^s T_{2i-1})= \prod_{i=1}^s e_{2i} - \prod_{i=1}^s e_{2i-1} = \prod_{i=1}^{2s}x_i - \prod_{i=1}^{2s}x_i =0$$
	so $\prod_{i=1}^s T_{2i} - \prod_{i=1}^s T_{2i-1} \in J$.
\end{remark}

An application of our work is to directed graphs, so we shall also recall basic terminology for directed graphs. A \emph{digraph} $D = (Z, \vec{E})$ consists of a set $Z$ of distinct points, called the \emph{vertices}, and a collection $\vec{E}$ of ordered pairs of vertices, called the \emph{directed edges}. We will also assume that all digraphs in this paper are simple digraphs. We shall write $x \rightarrow y$ for the directed edge from $x$ to $y$ in a digraph.

Directed walks, paths, circuits and cycles in a digraph can be defined similarly to those in an undirected graph with only one difference, that is, if $x_1, e_1, x_2, \dots, e_{s-1}, x_s$ represents a directed walk from $x_1$ to $x_s$ then $e_i$ is the directed edge $x_i \rightarrow x_{i+1}$ for all $i = 1, \dots, s-1$.

The application of our work to directed cycles in digraphs is based on the following construction (\cite{Vpersonal}).

\begin{definition} \label{def.Bgraph}
	Let $D = (Z, \vec{E})$ be a digraph over the vertex set $Z = \{z_1, \dots, z_m\}$. The bipartite graph $G(D)$, associated to $D$, is constructed as follows.
	\begin{enumerate}
		\item The vertices of $G(D)$ are $\{x_1, \dots, x_m, y_1, \dots, y_m\}$.
		\item The edges of $G(D)$ are:
		\begin{enumerate}
			\item $\{x_i,y_i\}$ for all $i = 1, \dots, m$; and
			\item $\{x_i,y_j\}$, for $i \not= j$, if $z_i \rightarrow z_j$ is an edge in $\vec{E}$.
		\end{enumerate}
	\end{enumerate}
\end{definition}

It is easy to see that for any digraph $D = (Z,\vec{E})$, the bipartite graph $G(D)$ has a perfect matching $\left\{e_i = \{x_i, y_i\} ~\big|~ i = 1, \dots, m\right\}$. We shall denote this perfect matching of $G(D)$ by $M_D$.

\begin{example}\label{ex.digraph}
Let $D$ be directed graph
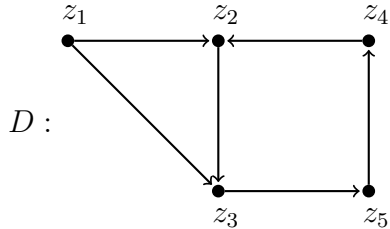
\begin{center}
	\begin{tikzpicture}[shorten >=1pt]
	
	\tikzstyle{point}=[circle,thick,draw=black,fill=black,inner sep=0pt,minimum width=4pt,minimum height=4pt]
		\node [label={[xshift=0cm, yshift=0cm]: $D:$}] at (0.5,8.5){};
	\node (a)[point,label={[xshift=0.1cm, yshift=0 cm]:$z_1$}] at (1,10) {};
	
	\node (b)[point,label={[xshift=0.1cm, yshift=0 cm]:$z_2$}] at (3,10) {};
	
	\node (c)[point,label={[xshift=0.1cm, yshift=-0.7 cm]:$z_3$}] at (3,8) {};
	
	\node (d)[point,label={[xshift=0.1cm, yshift=0 cm]:$z_4$}] at (5, 10) {};
	
	\node (e)[point,label={[xshift=0.1cm, yshift=-0.7 cm]:$z_5$}] at (5,8) {};

	\path[->, draw, thick]
	(a) edge (b)
	(b) edge (c)
	(a) edge (c)
	(d) edge (b)
	(e) edge (d)
	(c) edge (e);
\end{tikzpicture}
\end{center}

Then, the bipartite graph $G = G(D)$ associated to $D$ is
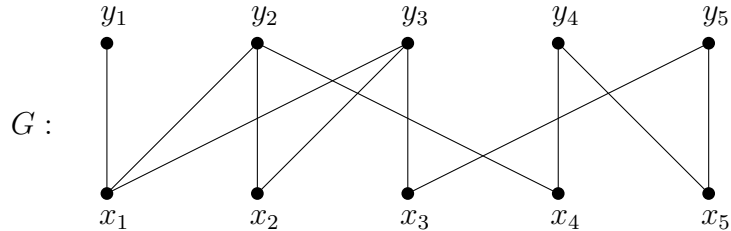
\begin{center}
	\begin{tikzpicture}[shorten >=1pt]
	
	\tikzstyle{point}=[circle,thick,draw=black,fill=black,inner sep=0pt,minimum width=4pt,minimum height=4pt]
	
	\node [label={[xshift=0cm, yshift=0cm]: $G:$}] at (6.5,8.5){};
	
	\node (x1)[point,label={[xshift=0.1cm, yshift=-0.7 cm]:$x_1$}] at (7.5,8) {};
	
	\node (x2)[point,label={[xshift=0.1cm, yshift=-0.7 cm]:$x_2$}] at (9.5,8) {};
	
	\node (x3)[point,label={[xshift=0.1cm, yshift=-0.7 cm]:$x_3$}] at (11.5,8) {};
	
	\node (x4)[point,label={[xshift=0.1cm, yshift= -0.7 cm]:$x_4$}] at (13.5,8) {};
	
	\node (x5)[point,label={[xshift=0.1cm, yshift=-0.7 cm]:$x_5$}] at (15.5,8) {};
	
	\node (y1)[point,label={[xshift=0.1cm, yshift=0 cm]:$y_1$}] at (7.5,10) {};

\node (y2)[point,label={[xshift=0.1cm, yshift=-0 cm]:$y_2$}] at (9.5,10) {};

\node (y3)[point,label={[xshift=0.1cm, yshift=-0 cm]:$y_3$}] at (11.5,10) {};

\node (y4)[point,label={[xshift=0.1cm, yshift=0 cm]:$y_4$}] at (13.5,10) {};

\node (y5)[point,label={[xshift=0.1cm, yshift=-0 cm]:$y_5$}] at (15.5,10) {};

\draw (x1.center) -- (y1.center);
\draw (x2.center) -- (y2.center);
\draw (x3.center) -- (y3.center);
\draw (x4.center) -- (y4.center);
\draw (x5.center) -- (y5.center);

\draw (x1.center) -- (y2.center);
\draw (x1.center) -- (y3.center);
\draw (x2.center) -- (y3.center);
\draw (x4.center) -- (y2.center);
\draw (x5.center) -- (y4.center);
\draw (x3.center) -- (y5.center);
	
\end{tikzpicture}
\end{center}
\end{example}


\section{Even circuits in graphs} \label{sec.evenwalk}

In this section, we present an algebraic algorithm to enumerate indecomposable even circuits in a graph. Recall that $G = (V,E)$ is a simple graph on the vertex set $V = \{x_1, \dots, x_n\}$ with $r = |E|$. For $I=I(G)$, fix $\phi$ to be the presentation matrix of $I = I(G)$ that results from the Taylor resolution of $I$, as in Definition \ref{def.Taylor}. For the remainder of the paper, $\phi$ will always refer to the Taylor presentation matrix unless otherwise noted.

We start with the following simple observation about $\phi$. Example~\ref{ex.twosquares} already illustrates the statements below, which are generally known but written here for ease of reference.

\begin{lemma} \label{lem.phi}
	If $G$ is a graph and $I=I(G)$, then the following statements hold.
	\begin{enumerate}
		\item The entries of $\phi$ are monomials in $\{x_1, \ldots , x_n\}$.
		\item Every column of $\phi$ has precisely two nonzero entries.
		\item The nonzero entries in each column of $\phi$ are either both linear or both quadratic.
		\item Every linear column of $\phi$ corresponds to a path of length $2$ in $G$ whose end-vertices are the nonzero entries of this column.
		\item Every quadratic column of $\phi$ corresponds to a pair of disjoint edges.
	\end{enumerate}
\end{lemma}

\begin{proof}
The assertions are straightforward from the construction of the Taylor resolution of $I(G)$. Note that in general, all relations on a set of monomials can be generated by pairwise relations (i.e., relations of two monomials). If $m_1, m_2$ are monomials, then a minimal relation between them has the form $ am_1 + bm_2 = 0$ where $a=(lcm(m_1,m_2)/m_1)$ and $b=(lcm(m_1,m_2)/m_2)$ are monomials. If $m_1 \not= m_2$ both have degree $2$, then $a=m_2, b=m_1$ if $m_1, m_2$ have disjoint support. Otherwise, $a,b$ both have degree one. The results follow.
\end{proof}

As above, we denote by $B(\phi)$ and $\overline{B(\phi)}$ the Jacobian dual and respectively the reduced Jacobian dual of $\phi$. We obtain an immediate corollary of Lemma \ref{lem.phi} when $\phi$ is assumed to be the Taylor presentation matrix of $I(G)$ for a graph $G$.

\begin{corollary} \label{cor.reducedB}
The nonzero columns of $\overline{B(\phi)}$ are precisely the columns of $B(\phi)$ that correspond to linear columns of $\phi$. Particularly, each nonzero column of $\overline{B(\phi)}$ contains precisely two nonzero entries, each of which is a degree one monomial in $T_1, \ldots , T_r$.
\end{corollary}

\begin{proof} By definition, the nonzero entries of $\overline{B(\phi)}$ come from nonzero entries of $B(\phi)$ that are contained in $\kk[T_1, \ldots, T_r]$. Observe, by the equation
\begin{align}
[T_1 \ \dots \ T_r] \cdot \phi = [x_1 \ \dots \ x_n] \cdot B(\phi), \label{eq.Bphi}
\end{align}
that the columns of $B(\phi)$ correspond to columns of $\phi$.
Moreover, by Lemma \ref{lem.phi}, the nonzero entries in each column of $\phi$ are of the same degrees (either linear or quadratic). It further follows from the equation (\ref{eq.Bphi}) that the degree with respect to the $x_i$'s of nonzero entries of a column in $B(\phi)$ is exactly one less than that of the corresponding column of $\phi$. Hence, nonzero columns of $\overline{B(\phi)}$ correspond to columns without the $x_i$'s in $B(\phi)$, which correspond to columns of linear forms (and 0) in $\phi$.

The second statement also follows from Lemma \ref{lem.phi}.
\end{proof}

\begin{remark}
Note that since zero columns of a matrix will not play any role in what follows, we could define $\overline{B(\phi)}$ to exclude all its zero columns. That is, we are working just with the (uniquely defined) columns of $\phi$, whose nonzero entries are linear, that result from binomial relations of edges in paths of length 2 in $G$. As mentioned earlier, zero rows of $\overline{B(\phi)}$ will not play a role and can be eliminated by using the content ideal of $\phi$ to define $B(\phi)$ rather than $\mm$. However, since zero rows do not affect minors, which are our main focus when using $\overline{B(\phi)}$, it is a matter of convenience to allow them.
\end{remark}

As stated in Lemma~\ref{lem.phi}, the linear columns of $\phi$ are generated by pairs of monomials corresponding to edges that share a vertex. In other words, the linear columns of $\phi$ correspond to paths of length $2$ in the graph. It can be desirable for computational purposes to use a minimal presentation matrix for $\phi$ rather than the full Taylor presentation matrix. It is easy to check that there are three paths of length two in each triangle, yielding three linear relations, any two of which generate the third. Since this is the only redundancy among the linear relations for a graph, if the graph is triangle free, the linear columns of a minimal presentation matrix will be the same as the linear columns of the Taylor presentation matrix.

Since the linear columns arise from paths of length two, as seen in Lemma~\ref{lem.phi} the endpoints of each path are the nonzero entries of that column of $\phi$. These endpoints will thus be encoded in the corresponding column of $B(\phi)$ as the rows in which the nonzero entries appear. It is natural to expect that the third vertex, the midpoint of the path, would  play a role.

\begin{definition} We call two nonzero columns of $\overline{B(\phi)}$ \emph{center-distinct} if their corresponding paths of length 2 in $G$ have distinct middle vertices. We also call the middle vertices of these paths of length 2 the \emph{mid-points} of the corresponding columns.
\end{definition}

Finding the mid-point of a column of $\overline{B(\phi)}$ can be done easily by examining the corresponding edges of $G$. If $T_i$ and $T_j$ are the two nonzero entries of a column of $\overline{B(\phi)}$ and $f_i, f_j$ are the corresponding edges of $g$ (that is, $\theta(T_i) = f_it$,and $\theta(T_j)=f_jt)$), then the mid-point of the column is ${\mbox{\rm supp}}f_i \cap {\mbox{\rm supp}}f_j$, or equivalently $gcd(f_i,f_j)$. 

The next lemma collects information that can be gleaned about a graph from minors of $\overline{B(\phi)}$ of a form that will appear frequently in the remainder of the article.

\begin{lemma}
	\label{minor.form} 
	If $\overline{B(\phi)}$ has a minor of the form 
	\begin{align}
	M =  \left[\begin{array}{cccc}
	T_2 & 0 & \ldots & -T_{2t-1} \\
	-T_1 & T_4 & \ldots & 0 \\
	0 & -T_3 & \ldots & 0 \\
	\vdots & \vdots & \ldots& \vdots \\
	0& 0& \ldots & T_{2t}
	\end{array}\right], \label{eq.M}
	\end{align}
	then $G$ contains an even closed walk corresponding to $\det M$. Moreover, 
	\begin{enumerate}
		\item the walk is primitive if and only if the columns of $M$ are pairwise center-distinct; and
		\item the walk is a circuit if and only if the nonzero entries of $M$ are distinct, in which case the circuit is indecomposable if and only if the columns of $M$ are pairwise center-distinct.
	\end{enumerate}
\end{lemma}

\begin{proof}
Combining equation (\ref{eq.Bphi}) with Lemma \ref{lem.phi} gives that each column of $M$ corresponds to a path of length 2 in $G$ whose end-vertices are labeled by the rows of $M$ corresponding to the nonzero entries in that column. By re-indexing the variables, we may assume that the rows of $M$ correspond to the variables $x_1, \dots, x_t$. Then, the $i$th column of $M$, for $1 \le i \le t-1$, corresponds to a path of length 2 from $x_i$ to $x_{i+1}$, and the last column of $M$ corresponds to a path of length 2 from $x_t$ to $x_1$. We shall denote those paths by $x_i, y_i, x_{i+1}$, for $i = 1, \dots, t-1$, and $x_{t}, y_{t}, x_1$. Furthermore, edges on these paths correspond to the variables $T_1, \dots, T_{2t}$. Hence, these paths glue together to form an even closed walk of length $2t$ in $G$. Since $\det M =\prod_{i=1}^t T_{2i} - \prod_{i=1}^t T_{2i-1}$ we have that $\det M$ corresponds to an even closed walk in $G$, as in Remark~\ref{cycles.eqtns}. This walk is a circuit if and only if the $T_i$, and thus the corresponding edges, are distinct. Finally, the columns of $M$ are pairwise center-distinct if and only if the vertices $y_1, \dots, y_t$ are pairwise distinct. Note that the vertices $x_1, \ldots, x_t$ are distinct by definition. This guarantees that the obtained closed walk or circuit of length $2t$ in $G$ is primitive or indecomposable respectively if and only if the columns of $M$ are pairwise center-distinct.	
\end{proof}

We are now ready to state the main result of the paper. Note that relabeling the vertices or edges of a graph corresponds to rearranging the rows of $\phi$ or of $B(\phi)$. Such a rearrangement will not affect the minors of a matrix, so when convenient, a specific labeling of vertices can be used without loss of generality.

\begin{theorem}\label{thm.main}
Let $G$ be a graph. Then the indecomposable even circuits of $G$ correspond exactly to the binomial minors of $\overline{B(\phi)}$ which satisfy the following conditions:
\begin{enumerate}
\item the monomials in these binomials are square-free and relatively prime; and
\item the columns of the corresponding submatrices are pairwise center-distinct.
\end{enumerate}
\end{theorem}

\begin{proof}
Suppose that $C$ is an indecomposable even circuit in $G$. For ease of notation, select a labeling on the vertices and edges so that the edges of $C$ (in order) are $e_1, \dots, e_{2t}$, where $e_i=x_ix_{i+1}$ for $i<2t$ and $e_{2t}=x_{2t}x_1$. Since $C$ is indecomposable, it is easy to see that $x_1, x_3, \dots, x_{2t-1}$ are pairwise distinct. Particularly, the linear relations of $I = I(G)$ include $x_1e_2-x_3e_1, x_3e_4-x_5e_3, \ldots , x_{2t-1}e_{2t}-x_1e_{2t-1}$ which correspond to the following columns of $\phi$:
$$\left[\begin{array}{cccc}
	-x_3 & 0 & \ldots & 0  \\
	x_1 & 0 & \ldots & 0 \\
	0 & -x_5 & \ldots & 0 \\
	0 & x_3 & \ldots & 0  \\
	\vdots & \vdots & \ldots & \vdots \\
	\vdots & \vdots & \ldots & -x_1 \\
	0& 0 & \ldots & x_{2t-1} \\
	\vdots & \vdots & \ldots & \vdots
	\end{array}\right]$$
where for convenience, labelings were chosen so that $T_i$ corresponds to $e_i$ for $1 \leq i \leq 2t$.

We can reorder the columns of $\phi$ so that these are the first $t$ columns. These columns produce  $x_1T_2-x_3T_1, x_3T_4-x_5T_3, \ldots , x_{2t-1}T_{2t}-x_1T_{2t-1}$  as linear equations of the Rees algebra $R[It]$, which correspond to the first $t$ equations of $[x_1 \ \dots \ x_n]\cdot B(\phi)$. Thus, the first $t$ columns of $B(\phi)$ are:
	$$\left[\begin{array}{cccc}
	T_2 & 0 & \ldots & -T_{2t-1} \\
	0 & 0 &  \ldots & 0 \\
	-T_1 & T_4 & \ldots & 0  \\
	0 & 0 & \ldots & 0\\
	0 & -T_3 & \ldots & 0 \\
	\vdots & \vdots & \ldots & \vdots \\
	0& 0& \ldots & T_{2t} \\
		\vdots & \vdots & \ldots & \vdots
	\end{array}\right].$$
By Corollary \ref{cor.reducedB}, these columns of $B(\phi)$ are unchanged when passing to $\overline{B(\phi)}$. Consider the $t \times t$ submatrix $M$ of $\overline{B(\phi)}$ consisting of the first $t$ columns and the $t$ identified nonzero rows:
$$\begin{array}{cccc}
M =  \left[\begin{array}{cccc}
	T_2 & 0 & \ldots & -T_{2t-1} \\
	-T_1 & T_4 & \ldots & 0 \\
	0 & -T_3 & \ldots & 0 \\
	\vdots & \vdots & \ldots& \vdots \\
	0& 0& \ldots & T_{2t}
	\end{array}\right]. 
	\end{array}$$
Then, $\det(M)=\prod_{i=1}^t T_{2i} - \prod_{i=1}^t T_{2i-1}$. This is a binomial whose monomials are square-free and relatively prime. Observe further that, since $C$ is indecomposable, $x_2, x_4, \dots, x_{2t}$ are pairwise distinct. Therefore, the columns of $M$ are pairwise center-distinct.

Conversely, suppose that $M$ is a $t \times t$ submatrix of $\overline{B(\phi)}$ whose determinant is a binomial of degree $t$ with square-free and relatively prime monomials, and whose columns are pairwise center-distinct. It follows from Lemma \ref{lem.phi} and Corollary \ref{cor.reducedB} that each column of $M$ contains at most 2 nonzero entries. Since the monomials in $\det(M)$ are relatively prime, each column of $M$ must contain exactly 2 nonzero entries. Particularly, $M$ contains exactly $2t$ nonzero entries.
Also, since the monomials in $\det(M)$ are relatively prime, $\det(M)$ contains exactly $2t$ distinct variables. Thus, all the $2t$ nonzero entries of $M$ are distinct. Since each row also contains at least $2$ distinct entries in order for the monomials to be relatively prime, a simple counting argument guarantees exactly two nonzero entries per row as well.

Now, by rearranging the rows and columns of $M$, it is easy to put $M$ in a block-matrix form, where each nonzero block is of the form as in (\ref{eq.M}) and lies on the diagonal. Observe further that if $M$ has more than one such block, then $\det(M)$ is not a binomial since all entries of $M$ are distinct. Therefore, we can assume that $M$ is exactly as in (\ref{eq.M}). The result now follows from Lemma~\ref{minor.form}.
\end{proof}

Theorem \ref{thm.main} gives us the following algebraic algorithm to detect the existence of and list all even circuits in a given graph $G$.

\begin{algorithm} \label{alg}
	To enumerate all indecomposable even circuits in a given graph $G$:
\begin{enumerate}
	\item Form $\phi$.
	\item Compute $\overline{B(\phi)}$.
	\item for $t$ from $1$ to the rank of $\overline{B(\phi)}$ compute all $t \times t$ submatrices of $\overline{B(\phi)}$.
	\item Test if each submatrix satisfies the following conditions:
\begin{enumerate}
	\item its columns are pairwise center-distinct; and
	\item its determinant is a binomial whose monomials are square-free and relatively prime.
\end{enumerate}
	\item If the answer is ``Yes'' then return the rows and centers of the columns corresponding to the minor. These are the vertices of an indecomposable even circuit in $G$.
\end{enumerate}
\end{algorithm}

\begin{remark}
	Note that only the linear columns of $\phi$ are necessary in this process and so in step 1 only the linear relations need be considered. Note also that the Taylor resolution and its presentation matrix can be constructed in polynomial time on the number of generators of $I(G)$ (i.e., the number of edges in $G$). Note further that the computation of the determinant of a matrix can also be done in polynomial time on the size of the matrix, and the rank of $\overline{B(\phi)}$ is at most the number of edges in $G$. Finally, testing if the columns of a minor in $\overline{B(\phi)}$ are center-distinct can be done in polynomial time on the size of the minor, which is bounded by the number of edges in $G$. Thus, Algorithm \ref{alg} runs in polynomial time on the size of the edge set of $G$.
\end{remark}

\begin{example} \label{ex.1}
	Consider the following graph.
	
	
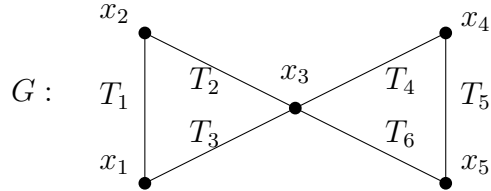
\begin{center}
		\begin{tikzpicture}
		
		\tikzstyle{point}=[circle,thick,draw=black,fill=black,inner sep=0pt,minimum width=4pt,minimum height=4pt]
		
		\node [label={[xshift=0cm, yshift=0cm]: $G:$}] at (1.5,0.8){};
		\node [label={[xshift=-0.4cm, yshift= -.10 cm]: $T_1$}] at (3,0.8){};
		\node [label={[xshift=-0.4cm, yshift= -.10 cm]: $T_2$}] at (4.2,1){};
		\node [label={[xshift=-0.4cm, yshift= -.10 cm]: $T_3$}] at (4.2,0.3){};
		\node [label={[xshift=-0.4cm, yshift= -.10 cm]: $T_4$}] at (6.8,1){};
		\node [label={[xshift=-0.4cm, yshift= -.10 cm]: $T_6$}] at (6.8,0.3){};
		\node [label={[xshift=0.4cm, yshift= -.10 cm]: $T_5$}] at (7,0.8){};
		
		\node (a)[point,label={[xshift=-0.4cm, yshift=-.10 cm]: $x_1$}] at (3,0) {};
		
		\node (b)[point,,label={[xshift=-0.4cm, yshift=-.10 cm]:$x_2$}] at (3,2) {};
		
		\node (c)[point,,label={[xshift=0cm, yshift=0.1 cm]: $x_3$}] at (5,1) {};
		\node(d)[point,,label={[xshift=0.4cm, yshift=-.20 cm]:$x_4$}] at (7,2) {};
		
		\node (e)[point,label={[xshift=0.4cm, yshift=-.10 cm]:$x_5$}] at (7,0) {};
		
		\draw (a.center) -- (b.center) -- (c.center) -- (a.center);
		\draw (c.center) -- (d.center) -- (e.center) -- (c.center);
		\end{tikzpicture}
	\end{center}
	
	The nonzero columns of the reduced Jacobian dual of $G$ corresponding to the Taylor presentation matrix $\phi$ are computed to be:
	$$\overline{B(\phi)}: \quad \quad \begin{array}{c} x_1 \\ x_2 \\ x_3 \\ x_4 \\ x_5 \end{array}
	\left[\begin{array}{cccccccccc} -T_2 & T_2 & 0 & 0 & 0 & 0 & 0 & 0 & -T_4 & -T_6 \\
	0 & -T_3& T_3 & 0 & 0 & 0 & -T_4 & -T_6 & 0 & 0 \\
	T_1 & 0 & -T_1 & -T_5 & T_5 & 0 & 0 & 0 & 0 & 0 \\
	0 & 0 & 0 & 0 & -T_6 & T_6 & T_2 & 0 & T_3 & 0 \\
	0 & 0 & 0 & T_4 & 0 & -T_4 & 0 & T_2 & 0 & T_3 \end{array} \right],$$
	where the labels $x_1, \dots, x_5$ indicate the variables of the corresponding rows. Furthermore, the mid-points of the columns are successively $x_2, x_3, x_1, x_4, x_5, x_3, x_3, x_3, x_3, x_3$.	
	By evaluating the minors of $\overline{B(\phi)}$, the only binomial minor whose monomials are square-free and relatively prime is $T_1T_4T_6 - T_2T_3T_5$, which corresponds to the only indecomposable even circuit $x_1, x_2, x_3, x_4, x_5, x_3, x_1$ in $G$. This minor appears using the submatrix formed by taking rows $1, 3, 4$ and columns $1, 5, 9$, for example, or the one formed by rows $2,3,4$ and columns $3,5,7$.
\end{example}

\begin{remark}
	There can be binomial minors of $\overline{B(\phi)}$ whose monomials are neither square-free nor relatively prime. These minors may correspond to even closed walks which transverse an edge multiple times.
\end{remark}

\begin{example} \label{ex.2}
	Consider the following graph.
	
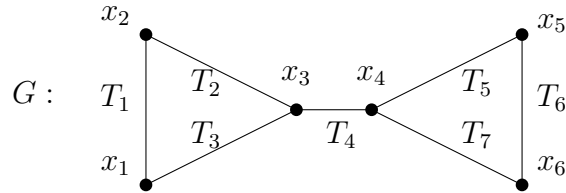
\begin{center}
		\begin{tikzpicture}
		
		\tikzstyle{point}=[circle,thick,draw=black,fill=black,inner sep=0pt,minimum width=4pt,minimum height=4pt]
		
		\node [label={[xshift=0cm, yshift=0cm]: $G:$}] at (1.5,0.8){};
		\node [label={[xshift=-0.4cm, yshift= -.10 cm]: $T_1$}] at (3,0.8){};
		\node [label={[xshift=-0.4cm, yshift= -.10 cm]: $T_2$}] at (4.2,1){};
		\node [label={[xshift=-0.4cm, yshift= -.10 cm]: $T_3$}] at (4.2,0.3){};
		\node [label={[xshift=-0.4cm, yshift= -0.8 cm]: $T_4$}] at (6,1){};
		\node [label={[xshift=-0.4cm, yshift= -.10 cm]: $T_5$}] at (7.8,1){};
		\node [label={[xshift=-0.4cm, yshift= -.10 cm]: $T_7$}] at (7.8,0.3){};
		\node [label={[xshift=0.4cm, yshift= -.10 cm]: $T_6$}] at (8,0.8){};
		
		\node (a)[point,label={[xshift=-0.4cm, yshift=-.10 cm]: $x_1$}] at (3,0) {};
		
		\node (b)[point,,label={[xshift=-0.4cm, yshift=-.10 cm]:$x_2$}] at (3,2) {};
		
		\node (c)[point,,label={[xshift=0cm, yshift=0.1 cm]: $x_3$}] at (5,1) {};
		\node (d)[point,,label={[xshift=0cm, yshift=0.1 cm]: $x_4$}] at (6,1) {};
		\node(e)[point,,label={[xshift=0.4cm, yshift=-.20 cm]:$x_5$}] at (8,2) {};
		\node (f)[point,label={[xshift=0.4cm, yshift=-.10 cm]:$x_6$}] at (8,0) {};
		
		\draw (a.center) -- (b.center) -- (c.center) -- (a.center);
		\draw (c.center) -- (d.center) -- (e.center) -- (f.center) -- (d.center);
		\end{tikzpicture}
	\end{center}
	
	The nonzero columns of the reduced Jacobian dual of $I(G)$ with respect to theTaylor presentation matrix $\phi$ are computed to be:
	$$\overline{B(\phi)}: \quad \quad \begin{array}{c} x_1 \\ x_2 \\ x_3 \\ x_4 \\ x_5 \\ x_6 \end{array}
	\left[\begin{array}{cccccccccc} -T_2 & T_2 & 0 & 0 & -T_4 & 0 & 0 & 0 & 0 & 0 \\
	0 & -T_3& T_3 & T_4 & 0 & 0 & 0 & 0 & 0 & 0 \\
	T_1 & 0 & -T_1 & 0 & 0 &-T_5 & -T_7 & 0 & 0 & 0 \\
	0 & 0 & 0 & -T_2 & T_3 & 0 & 0 & -T_6 & T_6 & 0 \\
	0 & 0 & 0 & 0 & 0 & T_4 & 0 & 0 & -T_7 & T_7 \\
	0 & 0 & 0 & 0 & 0 & 0 & T_4 & T_5 & 0 & -T_5 \end{array} \right],$$
	where the labels $x_1, \dots, x_6$ indicate the variables of the corresponding rows. Furthermore, the mid-points of the columns are successively $x_2, x_3, x_1, x_3, x_3, x_4, x_4, x_5, x_6, x_4$.
	
	The only binomial minor of $\overline{B(\phi)}$ is $T_2T_3T_5T_7 - T_1T_4^2T_6$. This corresponds to submatrices formed using rows $1,3,4,5$ and columns $1,5,6,9$ for example, or rows $2,3,4,6$ and columns $3,4,7,8$. A monomial of this minor is not square-free. This indicates that $G$ contains an even closed walk $x_2, x_3, x_4, x_5, x_6, x_4, x_3, x_1, x_2$, but this walk is not a circuit because it transverses through the edge $T_4$ twice. Hence, $G$ has no indecomposable even circuits.
\end{example}

\begin{remark} \label{rmk.cycle}
	With basically the same proof as that of Theorem \ref{thm.main}, it can be shown that the even cycles of $G$ correspond exactly to the binomial minors of $\overline{B(\phi)}$ which satisfy the following conditions:
	\begin{enumerate}
		\item the monomials in these binomials are square-free and relatively prime; and
		\item the variables labeling the rows and the mid-points of the columns of the corresponding submatrices are pairwise distinct.
	\end{enumerate}
\end{remark}

\begin{example} Let $G$ be the graph in Example \ref{ex.1}. As shown in Example \ref{ex.1}, the only binomial minor of $\overline{B(\phi)}$ whose monomials are square-free and relatively prime is $T_1T_4T_6 - T_2T_3T_5$. This minor is obtained by taking rows 1, 3, 4 and columns 1, 5, 9. In this minor, the mid-points of the columns are $x_2, x_5$ and $x_3$. On the other hand, the variables labeling the rows are $x_1, x_3$ and $x_4$. Clearly, we have a repeated $x_3$ among the mid-points and the row labels. Thus, this minor corresponds to an even indecomposable circuit, which is not a cycle.
\end{example}

\begin{example} Let $G$ be the graph in Example \ref{ex.twosquares}. A binomial minor of $\overline{B(\phi)}$ whose monomials are square-free and relatively prime is $T_2T_5T_7-T_1T_3T_6$. This minor is obtained by taking rows 1, 3, 5 and columns 1, 5, 9. The mid-points of the columns are $x_2, x_4$ and $x_6$, and the variables labeling the rows are $x_1, x_3$ and $x_5$. Since these are distinct variables, this minor corresponds to an even cycle $x_1, x_2, x_3, x_4, x_5, x_6, x_1$.
\end{example}

We continue this section by discussing an algebraic consequence of Theorem \ref{thm.main} in finding defining equations for the Rees algebras of edge ideals of graphs. Recall that the Rees algebra $R[It]$ of $I$ has a presentation $R[T_1, \dots, T_r] \stackrel{\theta}{\rightarrow} R[It] \rightarrow 0$, and $J = \text{ker } \theta$ is called the defining ideal of $R[It]$.

It was shown in \cite{Rafael1995} that the nonlinear equations of $J$ arise from the even closed walks in the graph $G$. An alternate proof of this fact appears in Chapter 10.1 of \cite{HH}. Also, it was proved in \cite[Corollary 10.1.5]{HH} that the generators correspond to primitive even closed walks and form a reduced Gr\"{o}bner basis for $J$. The binomials corresponding to indecomposable even circuits of $G$ are thus known to be elements of $J$. However, there are elements of $J$ that do not correspond to indecomposable circuits, as seen in Example \ref{ex.2}. It is worth noting that it was established in a more general setting that the maximal minors of $\overline{B(\phi)}$ are contained in $J$ (see, for example, \cite{UV}).
A close examination of the proof of Theorem~\ref{thm.main} shows that any even closed walk in $G$ corresponds to a binomial minor of $\overline{B(\phi)}$.

\begin{corollary} \label{cor.J}
Let $G$ be a graph with edge ideal $I = I(G)$, and let $J$ be the defining ideal of $R[It]$. Then, every nonlinear generator of $J$ appears as a binomial minor of $\overline{B(\phi)}$.
\end{corollary}

\begin{proof}
By \cite{Rafael1995} and \cite[Chapter 10.1]{HH}, we have that the nonlinear generators of $J$ correspond to primitive even closed walks in $G$. Consider any primitive even closed walk $W$ in $G$ and, after a re-labeling, suppose that the vertices on $W$ are $x_1, \dots, x_{2t}$ (not necessarily distinct).

Observe that, since $W$ is primitive, if $x_i=x_j$ is a repeated vertex in $W$ then $i$ and $j$ are of different parity. View $W$ as the union of $t$ paths of length 2, namely, $P_i = x_{2i-1}, x_{2i}, x_{2i+1}$, for $i = 1, \dots t$ (where $x_{2t+1} = x_1$). Then, the endpoints of each path $P_i$ are distinct vertices. Thus, $P_i$ corresponds to a column of $\overline{B(\phi)}$ with exactly two nonzero entries, appearing in the rows labeled by $x_{2i-1} \not= x_{2i+1}$. Selecting these columns and the corresponding nonzero rows results in a $t \times t$ submatrix $M_W$ of $\overline{B(\phi)}$.

As in the proof of Theorem~\ref{thm.main}, the rows and columns of $M_W$ can be rearranged so that $M_W$ is a block-matrix in which each block is of the form of (\ref{eq.M}). If there are multiple blocks then each corresponds to an even closed walk contained in $W$ by Lemma~\ref{minor.form}, a contradiction to the fact that $W$ is primitive. Therefore, $M_W$ is of the form (\ref{eq.M}), where the nonzero entries may not be distinct. Hence, the corresponding generator of $J$ is the same as $\det(M_W)$, which is a binomial minor of $\overline{B(\phi)}$.
\end{proof}

Let $\T = [T_1 \ \dots \ T_r]$.  Corollary \ref{cor.J} gives us the containment
$$J \subseteq \big\langle \T \cdot \phi, I_2(\overline{B(\phi)}), I_3(\overline{B(\phi)}), \dots ,I_k(\overline{B(\phi)})\big\rangle,$$
where $k$ is the rank of $\overline{B(\phi)}$. The reverse containment fails to hold. In general, $I_t(\overline{B(\phi)})$ will contain monomials that are not in $J$. For instance, in Example \ref{ex.twosquares}, $T_2T_3 \in I_2(\overline{B(\phi)})$ but $T_2T_3 \not\in J$. Interestingly enough, we shall see that by restricting to binomial minors we in fact obtain an equality. While not all binomial minors of $\overline{B(\phi)}$ are minimal generators of $J$, such minors correspond to multiples of binomials which come from (not necessarily primitive) even closed walks and are elements of $J$. The following lemma will be used in proving the desired equality. For convenience, we consider $1$ to be a trivial monomial.

\begin{lemma}\label{rearrangements}
	Suppose $\psi$ is an $n \times n$ matrix such that:
	\begin{enumerate}
		\item $\det (\psi)$ is a nonzero binomial;
		\item every column of $\psi$ has at most $2$ nonzero entries.
	\end{enumerate}
Then, after row and column rearrangements, $\psi$ has a block decomposition
$$\psi = \left[ \begin{array}{cc}
X & W \\Y & Z
\end{array} \right]$$
where $\det (X)$ is a monomial, $\det (Z)$ is a binomial, $\det (\psi) = \det(X) \det (Z)$, and every row of $Z$ has at least $2$ nonzero entries and every column of $Z$ has exactly $2$ nonzero entries.
\end{lemma}

\begin{proof}
	If every row of $\psi$ has at least $2$ nonzero entries and every column of $\psi$ has exactly $2$ nonzero entries, set $X, Y, W$ to be empty matrices and $Z= \psi$. Then since an empty product is defined to be $1$, $\det (X)=1$ is a (degenerate) monomial, and the result holds. In particular, the result holds when $n=2$. Assume $n>2$.
	
	Assume there exist $s$ rows of $\psi$ with a single nonzero entry.
	Since $\det (\psi) \not= 0$, every row and column of $\psi$ has at least one nonzero entry and no two rows (columns) have a single nonzero entry in the same column (row). Note also that row and column exchanges modify only the sign of the determinant and not the binomial nature.  By performing row exchanges, we can rearrange all rows with a single nonzero entry to come before all rows with multiple nonzero entries. That is, we may assume that $\psi$ has the form
	$$\psi = \left[\begin{array}{c|c}
	D_1 & 0\\
	\hline
	A_1 & \psi_1
	\end{array}\right],$$
	where $D_1$ is an $s \times s$ diagonal matrix, and $A_1$ is a matrix where each column has at most one nonzero entry. Observe that $\det (D_1)$ is a monomial and $\det (\psi) = \det (D_1) \det(\psi_1)$, so $\det (\psi_1)$ is again a nonzero binomial. As before, each column of $\psi_1$ has either one or two nonzero entries and each row has at least one nonzero entry. Since $\psi_1$ is $n-s \times n-s$ with $s \geq 1$, by induction, 
	$$\psi_1 = \left[ \begin{array}{cc}
	X_1 & W_1 \\Y_1 & Z_1
	\end{array} \right]$$
	where $\det (X_1)$ is a monomial, $\det (Z_1)$ is a binomial, $\det (\psi_1) = \det (X_1) \det (Z_1)$, and every row of $Z_1$ has at least $2$ nonzero entries and every column of $Z_1$ has exactly $2$ nonzero entries. 
	Now 
		$$\psi =	\left[\begin{array}{c|c}
	D_1 & 0\\
	\hline
	A_1 & \begin{array}{c|c}
	X_1 & W_1 \\
	\hline
	Y_1 & Z_1
	\end{array}
	\end{array}\right] = \left[\begin{array}{c|c|c}
	D_1 & 0 & 0\\
	\hline
	A_1' & X_1 & W_1\\
	\hline
	A_1'' & Y_1 & Z_1
	\end{array}\right],$$
	where $A_1', A_1''$ consist of the appropriate rows of $A_1$. Set $Z=Z_1$, $X=\left[ \begin{array}{c|c}
				D_1 & 0 \\
				\hline
				A_1' & X_1
				\end{array} \right]$,
	$W=\left[\begin{array}{c} 0 \\ \hline W_1 \end{array} \right]$, and $Y=\left[ \begin{array}{c|c} A_1'' & Y_1 \end{array} \right]$. Note that 
	$$\det(\psi) = \det(D_1) \det (\psi_1) = \det(D_1) \det (X_1) \det (Z_1) = \det (X) \det (Z)$$ and the result follows.

	Similarly, if any column of $\psi$ has a single nonzero entry, then by performing column exchanges, we may assume $\psi$ has the form
	$$\psi = \left[\begin{array}{c|c}
	D_2 & B_2\\
	\hline
	0 & \psi_2
	\end{array}\right],$$
	where $D_2$ is a diagonal matrix. Observe that $\det (\psi) = \det (D_1) \det(\psi_2)$ and $\det (D_2)$ is a monomial, so $\det (\psi_2)$ is again a nonzero binomial. As before, each column of $\psi_2$ has either one or two nonzero entries and each row has at least one nonzero entry and the result follows by induction as in the case above.
	\end{proof} 

Note that in the above lemma, since the columns of $Z$ have $2$ nonzero entries each and the columns of $\psi$ have at most $2$ nonzero entries, it follows that $W=0$.
In order to state our next result formally, for a matrix $M$, set $\bi(I_t(M))$ to be the collection of $t \times t$ minors of $M$ that are binomials.

\begin{theorem}\label{thm.J}
Let $G$ be a graph with edge ideal $I = I(G)$. Let $J$ be the defining ideal of $R[It]$ and let $k = \rank \overline{B(\phi)}$. Then,
$$J = \big\langle \T \cdot \phi, \bi(I_2(\overline{B(\phi)})), \bi(I_3(\overline{B(\phi)})), \dots ,\bi(I_k(\overline{B(\phi)}))\big\rangle.$$
\end{theorem}

\begin{proof}
One inclusion follows directly from Corollary~\ref{cor.J}.

For the reverse inclusion, suppose that $\psi$ is a $t \times t$ submatrix of $\overline{B(\phi)}$ with $\det (\psi)$ a binomial. We will show that $\det (\psi) \in J$.
Indeed, since $\det (\psi) \not= 0$, every row and column of $\psi$ has at least one nonzero entry and no two rows (columns) have a single nonzero entry in the same column (row). As noted before, each nonzero column of $\overline{B(\phi)}$ has precisely $2$ nonzero entries. Thus, each column of $\psi$ has at most $2$ nonzero entries. Applying Lemma~\ref{rearrangements} we can assume that $\psi = \left[ \begin{array}{c|c} A & B \\ \hline C & \psi_p \end{array} \right]$ where $ \psi_p$ is a minor of $\psi$ in which every column has exactly two nonzero entries, every row has at least two nonzero entries, and $\det (\psi_p)$ is a nonzero binomial with $\det(\psi)= \det(A)\det(\psi_p)$. Thus, if $\det(\psi_p) \in J$, then $\det(\psi)\in J$.

Now, we can reorder the rows of $\psi_p$ so that the nonzero entries of the first column appear in the first two rows. Since the second row has at least two nonzero entries, we can rearrange the remaining columns of $\psi_p$ so that the $(2,2)$ entry is not zero. If the second nonzero entry of column $2$ is not in row $1$ then we can rearrange the remaining rows so that the $(3,2)$ entry of $\psi_p$ is not $0$. We can continue to rearrange the rows and columns of $\psi_p$ in this manner (see also (\ref{eq.M})) until for some row $i$, the remaining columns with nonzero entries in row $i$ have the second nonzero entry in row $j$ for some $j < i$. At this point, $\psi_p$ has the following form
$$\psi_p = \left[\begin{array}{c|c}
\begin{array}{cccccc}
* & 0 & 0 & 0 & \cdots & 0 \\
* & * & 0 & 0 & \cdots & 0 \\
0 & * & * & 0 & \cdots & *\\
0 & 0 & * & * & \cdots & 0 \\
 & & & \ddots & & \\
 0& 0& 0 & 0 & \cdots & *
\end{array} & N \\
\hline
0 & \psi_{p+1}
\end{array}\right],$$
where $*$ denotes a nonzero entry and the position of the second $*$ in the final column before $N$ is illustrative. Notice that $\psi_p$ has a minor of the form of (\ref{eq.M}). In addition, $\psi_p$ has the form
$$\psi_p = \left[\begin{array}{c|c}
\begin{array}{c|c}
L_{\psi_p} & 0\\
\hline
C_{\psi_p} & M_{\psi_p}
\end{array} & N \\
\hline
0 & \psi_{p+1}
\end{array}\right]$$
where $L_{\psi_p}$ is lower triangular (the empty matrix if $j=1$), and $M_{\psi_p}$ has the form of (\ref{eq.M}). Note that
$$\det(\psi_p) = \det(L_{\psi_p}) \det(M_{\psi_p}) \det(\psi_{p+1}).$$
By Lemma~\ref{minor.form}, $\det(M_{\psi_p})\in J$, and since $J$ is closed under multiplication, $\det(\psi)\in J$ as required.
\end{proof}

\begin{remark}
It can be seen in the proof above that every row of $\psi_{p+1}$ has at least two nonzero entries and each column has either one or two nonzero entries. We can continue the process to get a block decomposition of $\psi$, in which the determinant is the product of the determinants of the diagonal blocks, each of which is either a monomial or a binomial coming from a block of the form of (\ref{eq.M}). Hence, $\det(\psi)$ can be written as a product of monomials and of binomials coming from blocks of the form of (\ref{eq.M}). Since minors of the form of (\ref{eq.M}) correspond to (not necessarily primitive) even closed walks, these binomials are in  $J$ by \cite{Rafael1995}. Thus not only are the binomial minors in $J$, if they are not irreducible, they factor as products of monomials and binomial elements of $J$.
\end{remark}

In the proof of Theorem~\ref{thm.J}, the fact that $\det(\psi)$ was a binomial came from the statement. Since the entries of $\psi$ are not assumed to be distinct, it is possible for the product of two binomials to be a binomial. However, much of the proof holds if instead $\det (\psi)$ is assumed not to be a monomial. The following result demonstrates how to use this to obtain all nonlinear generators of $J$ from a single sized ideal of minors of $\overline{B(\phi)}$.

\begin{corollary}
Let $k = \rank \overline{B(\phi)}$. Then, all nonlinear generators of $J$ can be obtained as factors of generators of $I_k(\overline{B(\phi)})$.
\end{corollary}

\begin{proof}
By Corollary~\ref{cor.J}, every nonlinear generator of $J$ appears as some binomial minor of $\overline{B(\phi)}$. Let $f$ be a nonlinear generator of $J$ and let $M$ be the corresponding submatrix of $\overline{B(\phi)}$. Since $J$ is generated by primitive even walks, the monomials in $f$ are relatively prime, so every column of $M$ contains exactly two nonzero entries.  By performing row and column exchanges, write
$$\overline{B(\phi)} = \left[ \begin{array}{c|c}
M & * \\
\hline
0 & B_2
\end{array} \right].$$
The set of columns of $\overline{B(\phi)}$ used to form $M$ can be extended to a set of columns of full rank. That is, by selecting appropriate columns and rows from $B_2$, there is a $k \times k$ submatrix of $\overline{B(\phi)}$ of the form
$$\widehat{M} = \left[ \begin{array}{c|c}
M & * \\
\hline
0 & B_2'
\end{array} \right]$$
with nonzero determinant. Since $\det (\widehat{M}) = \det (M) \det (B_2')$, $f$ is a factor of an element of $I_k(\overline{B(\phi)})$ as desired.
\end{proof}


\section{Directed cycles in digraphs} \label{sec.digraph}

In this part of the paper, we shall discuss an interesting application of our main result, Theorem \ref{thm.main}, to the problem of detecting the existence of directed cycles in a given directed graph.

Let $D = (Z, \vec{E})$ be a directed graph over the vertex set $Z = \{z_1, \dots, z_m\}$. Let $G = G(D)$ be the undirected bipartite graph constructed from $D$ as in Definition \ref{def.Bgraph}. Recall that $M_D$ represents the perfect matching $\left\{e_i = x_iy_i ~\big|~ i = 1, \dots, m\right\}$ in $G = G(D)$.

The connection between directed cycles in a digraph $D$ and even cycles in $G(D)$ is well established (see for example \cite{DM, KL}). We shall present the following known result in a form that is convenient for applying Theorem~\ref{thm.main}. Note that, in an even circuit, we can start at any place and collect the first, third, fifth, etc. (all the odd ordered) edges, or collect the second, fourth, etc. (all the even ordered) edges. This way we will get two sets of disjoint edges, each consisting of exactly half the number of edges on the circuit. We will refer to each of these two sets a \emph{collection of alternating edges} in the even circuit.

\begin{theorem} \label{thm.cycle}
	The directed cycles in $D$ correspond exactly to the even cycles in $G = G(D)$ in which a collection of alternating edges forms a subset of the perfect matching $M_D$ in $G$.
\end{theorem}

\begin{proof} We start the proof with the following observation. Consider an even circuit $C$ in $G$ with the property that its alternating edges form a subset of the perfect matching $M_D$. It can be seen that if one transverses around $C$ on its edges and hits $x_i$ (or $y_i$) from an edge that is not in $M_D$ then the next edge of $C$ has to be $x_iy_i$. Since in a circuit the edges are distinct, this ensures that $C$ cannot contain $x_i$ (or $y_i$) more than once. That is, $C$ is a cycle (which is necessarily indecomposable). Thus, the indecomposable circuits in $G$ with the property that a collection of their alternating edges forms a subset of the perfect matching $M_D$ are exactly the even cycles in $G$ with the same property.
	
Suppose that $z_{i_1} \rightarrow z_{i_2} \rightarrow \dots \rightarrow z_{i_t} \rightarrow z_{i_1}$ is a directed cycle in $D$. By the construction of $G$, it is easy to see that $x_{i_1}, y_{i_2}, x_{i_2}, y_{i_3}, x_{i_3}, \dots, x_{i_t}, y_{i_1}, x_{i_1}$ is an indecomposable even circuit in $G$. Moreover, the collection of even edges in this circuit is $\{e_{i_1}, \dots, e_{i_t}\}$, which is a subset of the perfect matching $M_D$.
	
Conversely, suppose that $G$ contains an indecomposable even circuit $C$ in which a collection of alternating edges form a subset of the perfect matching $M_D$. Since $G$ is bipartite, every edge in $G$ (and so any edge in $C$) connects a vertex $x_{i_j}$ to a vertex $y_{i_k}$. Thus, without loss of generality, we may assume that the circuit $C$ is of the form $x_{i_1}, y_{i_2}, x_{i_2}, y_{i_3}, x_{i_3}, \dots, x_{i_t}, y_{i_1}, x_{i_1}$. Since $C$ is a circuit and $C$ contains $\{e_{i_1}, \dots, e_{i_t}\}$, it follows that $i_1, \dots, i_t$ are distinct indices. By the construction of $G$ again, we have a directed cycle $z_{i_1} \rightarrow z_{i_2} \rightarrow \dots \rightarrow z_{i_t} \rightarrow z_{i_1}$ in $D$.
\end{proof}

\begin{example}
Let $D$ be the directed graph in Example \ref{ex.digraph} and let $G = G(D)$ be its associated bipartite graph.
It can be seen that $G$ has only one even cycle whose alternating edges form a subset of the perfect matching $M_D$, namely, $x_2, y_3, x_3, y_5, x_5, y_4, x_4, y_2, x_2$. This even cycle of $G$ corresponds to the directed cycle $z_2 \rightarrow z_3 \rightarrow z_5 \rightarrow z_4 \rightarrow z_2$ in $D$.
	
	Note that $z_1, z_2, z_3, z_1$ does not form a directed cycle in $D$ even though its undirected edges would form a triangle. This is reflected by the fact that there is no even cycle between $x_1, y_1, x_2, y_2, x_3, y_3$ in $G$. Furthermore, not all even cycles in $G$ would correspond to directed cycles in $D$. For instance, consider the even cycle $x_1, y_3, x_3, y_5, x_5, y_4, x_4, y_2, x_1$ in $G$. Neither collection of alternating edges of this cycle is a subset of the perfect matching $M_D$, and the this even cycle does not correspond to any directed cycle in $D$ ($z_1, z_3, z_5, z_4, z_2, z_1$ does not form a directed cycle in $D$).
\end{example}

As a corollary of Theorem \ref{thm.cycle}, we derive an algebraic algorithm to enumerate all directed cycles in a given digraph. Note that, by the proof of Theorem \ref{thm.cycle}, indecomposable even circuits of $G = G(D)$, in which a collection of alternating edges is a subset of $M_D$, are exactly the even cycles in $G$ with the same property.

\begin{corollary} \label{cor.digraph}
Let $D$ be a digraph and let $G = G(D)$ be its corresponding bipartite graph. Let $\phi$ be the presentation of the edge ideal $I = I(G)$ of $G$ that results from its Taylor resolution. Let $\overline{B(\phi)}$ be its reduced Jacobian dual. Then, the directed cycles of length $t$ in $D$ correspond exactly to the binomial $t \times t$ minors of $\overline{B(\phi)}$ that satisfy the following conditions:
\begin{enumerate}
\item their columns are pairwise center-distinct;
\item their monomials are square-free and relatively prime; and
\item one of these monomials is the product of variables that correspond to a subset of the perfect matching $M_D$.
\end{enumerate}
\end{corollary}

\begin{proof}
The assertion is a direct consequence of Theorem \ref{thm.main} and Theorem \ref{thm.cycle}.
\end{proof}

We recall the famous Caccetta-H\"aggkvist conjecture for directed cycles in digraphs (\cite{CH}).

\begin{conjecture}[Caccetta-H\"aggkvist] \label{conjCH}
	Let $D$ be a digraph on $n$ vertices. Let $\ell \in \NN$ and suppose that the outdegree of each vertex in $D$ is at least $\dfrac{n}{\ell}$. Then $D$ contains a directed cycle of length at most $\ell$.
\end{conjecture}

As a consequence of Theorems \ref{thm.main} and \ref{thm.cycle}, we are able to present a Jacobian dual matrix interpretation of the Caccetta-H\"aggkvist conjecture as follows. Note that every cycle is a primitive even closed walk, and that for a bipartite graph, the two notions coincide.

\begin{conjecture} \label{conj}
	Let $D$ be a digraph on $n$ vertices such that the outdegree of each vertex in $D$ is at least $\dfrac{n}{\ell}$. Let $\phi$ be the Taylor presentation matrix of $I(G(D))$. Then for some $q \leq \ell$, $I_q(\overline{B(\phi)})$ contains a binomial with square-free, relatively prime terms, one of which is a product of elements of $M_D$.
\end{conjecture}

By Theorems \ref{thm.main} and \ref{thm.cycle}, Conjectures \ref{conjCH} and \ref{conj} are equivalent. Conjecture \ref{conj} can also be rephrased using the language of Rees algebras by using Theorem~\ref{thm.J}.

\begin{conjecture} \label{conj,J}
Let $D$ be a digraph on $n$ vertices such that the outdegree of each vertex in $D$ is at least $\dfrac{n}{\ell}$. If $J$ is the defining ideal of the Rees algebra $R[I(G(D))t]$ then, for some $q \leq \ell$, $J$ has a binomial generator of degree $q$, that is square-free and has relatively prime terms, one of which is a product of elements of $M_D$.
\end{conjecture}

We conclude the paper with the observation that Conjecture \ref{conj} can be further translated into a problem in linear algebra.
Notice that if the outdegree of a vertex $z_i$ is at least $r$, then there are at least $r$ paths of length $2$ using the edge $x_iy_i$ with $x_i$ as the mid-point. The corresponding linear relations $y_iT_{j_i} - y_jT_i$ give specific information about $r$ columns of $\overline{B(\phi)}$ each of which has an element from the perfect matching. If $D$ has $m$ vertices, this yields $mr$ columns of $\overline{B(\phi)}$, each of which contains an element from the perfect matching, which form a fertile source of potential minors using submatrices of the form of (\ref{eq.M}) that would correspond to directed cycles in $D$.

\end{document}